\documentclass[reqno]{amsart}
\usepackage{amssymb}
\usepackage{hyperref}
 \newtheorem{thm}{Theorem}[section]
 \newtheorem{cor}[thm]{Corollary}
 
 \newtheorem{prop}[thm]{Proposition}
 \theoremstyle{definition}
 \newtheorem{defn}[thm]{Definition}
 \theoremstyle{remark}
 \newtheorem{rem}[thm]{Remark}
 \newtheorem{ex}[thm]{Example}
 \numberwithin{equation}{section}

\begin{document}

%
%
%
%
%
%
%
%
%

\title[ pseudo B-Weyl operators and generalized Drazin invertibility]
{On pseudo B-Weyl operators and generalized Drazin invertibility for operator matrices
}

\author[H. Zariouh]{H. Zariouh}

\address{%
Centre r\'egional des m\'etiers
de l'\'education et de la formation,\\
 B.P 458, Oujda, Morocco et\\
 Equipe de la Th\'eorie des Op\'erateurs\\Universit\'e Mohammed
 I,
 Facult\'e des Sciences\\
 D\'ept. de Math\'ematiques et Informatique,  Morocco.}

\email{h.zariouh@yahoo.fr}

\author{H. Zguitti}
\address{Department of Mathematics and Computer Science\\
Multidisciplinary Faculty of Nador\\
Mohammed First University, PO Box 300\\
 Selouane 62702 Nador, Morocco.}
\email{h.zguitti@ump.ma}
\subjclass{Primary 47A53, 47A10, 47A11.}
\keywords{Generalized Drazin spectrum, pseudo B-Weyl operator, single-valued extension
property, operator matrices}


\begin{abstract}
We introduce a new class which generalizes the class of B-Weyl operators. We say
 that $T\in L(X)$ is pseudo B-Weyl if $T=T_1\oplus T_2$ where $T_1$ is a Weyl
 operator and $T_2$ is a quasi-nilpotent operator. We show that the corresponding
pseudo B-Weyl spectrum $\sigma_{pBW}(T)$ satisfies the equality
  $\sigma_{pBW}(T)\cup[{\mathcal S}(T)\cap{\mathcal
S}(T^*)]=\sigma_{gD}(T);$ where $\sigma_{gD}(T)$ is the generalized
Drazin spectrum of $T\in L(X)$ and ${\mathcal S}(T)$  (resp.,
${\mathcal S}(T^*)$) is the set where $T$ (resp., $T^*$) fails to
have SVEP. We also investigate the generalized Drazin invertibility
of upper triangular operator matrices by giving sufficient
 conditions which assure that the generalized Drazin spectrum or the pseudo
 B-Weyl spectrum of an upper triangular operator matrices is the union of its diagonal entries spectra.
\end{abstract}

\maketitle

\section{Introduction and Preliminaries}

Let $X$ and $Y$ be Banach
spaces and let $L(X,Y)$ denote
 the algebra of all bounded linear operators from $X$ to $Y.$ We shall write
 $L(X)$ for the algebra $L(X,X).$ For $T\in L(X),$ by  $T^*,$ ${\mathcal N}(T),$ ${\mathcal R}(T),$ $\sigma(T),$
 $\sigma_l(T),$ $\sigma_r(T),$ $\sigma_p(T),$ $\sigma_{ap}(T)$ and $\sigma_s(T),$ we denote
 respectively, the adjoint of $T,$
  the null space, the
 range, the spectrum of $T,$  the left spectrum of $T$, the right spectrum of $T,$ the point spectrum of $T,$
 the approximate point spectrum of $T$ and the surjective spectrum of
 $T.$

 \smallskip
 A bounded linear
 operator $T\in L(X)$ is said to have the {\it single-valued
 extension property} (SVEP for short) at $\lambda\in\mathbb{C}$ if
  for every open neighborhood $U_\lambda$ of $\lambda,$ the constant function $f\equiv 0$ is the only
 analytic solution of the equation
 $(T-\mu I)f(\mu)=0\quad\forall\mu\in U_\lambda.$ We denote by   ${\mathcal S}(T)$
  the open set of $\lambda\in\mathbb{C}$ where $T$ fails to have  SVEP at $\lambda$,   and we say that  $T$ has  SVEP
  if
$ {\mathcal S}(T)=\emptyset.$ It is easy to see that ${\mathcal S}(T)\subset\sigma_p(T)$ (See \cite{LN} for more details about
this spectral property). According to \cite[Lemma3]{LV} we have
$$\label{surj}
\sigma(T)={\mathcal S}(T)\cup\sigma_s(T)$$ and in particular $\sigma_s(T)$ contains the topological boundary of ${\mathcal S}(T)$. Moreover, it is obvious
 that $T$ has  SVEP at every point $\lambda\in
\mbox{iso}\sigma(T).$ Henceforth, the symbol $\mbox{iso}\Lambda$ means  isolated points of a
given subset $\Lambda$ of $\mathbb{C}$ and $\mbox{acc}\Lambda$ denotes the set
of all points of accumulation of $\Lambda.$

\smallskip
$T\in L(X)$ is called an {\it upper semi-Fredholm} (resp., {\it
lower semi-Fredholm}) if ${\mathcal R}(T)$ is closed and
$n(T):=\dim{\mathcal N}(T)<+\infty$ (resp.,
$d(T):=\mbox{codim}{\mathcal R}(T)<+\infty$). If $T$ is either upper
or lower semi-Fredholm then $T$ is called a {\it semi-Fredholm}
operator. The {\it index} of a semi-Fredholm operator $T$ is defined
by $\mbox{ind}(T)=n(T)-d(T).$ $T\in L(X)$ is called a {\it Fredholm}
operator if both $n(T)$ and $d(T)$ are finite, and is called a {\it
Weyl} operator if it is a Fredholm of index zero. The {\it essential
spectrum} of $T$ is defined by $\sigma_e(T)=\{ \lambda \in
\mathbb{C}: T-\lambda I \mbox{ is not a Fredholm operator}\},$ and
the {\it Weyl spectrum} of $T$ is defined by  $\sigma_{W}(T)=\{
\lambda \in \mathbb{C}: T-\lambda I \mbox{ is not a Weyl
operator}\}.$ Let ${\mathcal F}(X)$ denote the ideal of finite rank
operators in $L(X).$ Then it is well known that
$$\sigma_{W}(T)=\displaystyle{\bigcap_{F\in {\mathcal
F}(X)}}\sigma(T+F).$$

Recall that an operator $T\in L(X)$ is said to be {\it semi-regular}, if ${\mathcal R}(T)$ is closed and
${\mathcal N}(T^n)\subseteq {\mathcal R}(T),$ for all $n\in {\mathbb {N}},$ see for example \cite{MM}. In addition, it was proved in \cite{K} that  given a semi-Fredholm operator $T\in L(X),$ there exist two closed $T$-invariant  subspaces $X_1,$ $X_2$ such that $X=X_1\oplus X_2,$ $T|X_{1}$  is  nilpotent and  $T|X{_2}$ is  semi-regular. This decomposition
is known as the {\it Kato decomposition}, and the operators satisfying these conditions, which were
characterized in \cite{lab}, are said to be the {\it quasi-Fredholm operators}.

\smallskip
Berkani gave a generalization of Fredholm operators as follows: for each nonnegative integer
 $n$ define $T_{[n]}$ to be the
restriction of $T$ to $\mathcal{R}(T^n)$ viewed as a map from $\mathcal{R}(T^n)$ into
$\mathcal{R}(T^n)$ (in particular $T_{[0]}=T$). If for some $n$, $\mathcal{R}(T^n)$ is
closed and $T_{[n]}$ is a Fredholm operator then $T$ is called a
{\it B-Fredholm} operator. $T$ is said to be a {\it B-Weyl} operator
if $T_{[n]}$ is a Fredholm operator of index zero (see
\cite{Berk2002}). The {\it B-Weyl spectrum} $\sigma_{BW}(T)$ of $T$
is defined by
$$\sigma_{BW}(T)=\{\lambda\in\mathbb{C} : T-\lambda I\mbox{ is not a
B-Weyl operator}\}.$$
On the other hand, and according to \cite[Proposition 2.6]{BE1}, a B-Fredholm operator is quasi-Fredholm; what is more, according to \cite[Theorem 2.7]{BE1}, if $T\in L(X)$ is B-Fredholm, then there exist
two closed $T$-invariant  subspaces $X_1,$ $X_2$ such that $X=X_1\oplus X_2,$ $T|X_{1}$  is  Fredholm  and  $T|X{_2}$ is  nilpotent (see also \cite[Theorem 7]{muller}).

\smallskip
An operator $T\in L(X)$ is said to be a {\it Drazin invertible} if
there exists a positive integer $k$ and an operator $S\in L(X)$ such
that $$ST=TS, \quad T^{k+1}S=T^k\mbox{ and } S^2T=S.$$
It is well known that $T$ is Drazin
invertible  if and only if $T=U\oplus V;$
where $U$ is an invertible operator and $V$  is a nilpotent one (see
\cite[Corollary 2.2]{Lay}).  The {\it Drazin spectrum} of $T\in
L(X)$ is defined by
$$\sigma_D(T)=\{\lambda\in\mathbb{C} : T-\lambda I\mbox{ is not
Drazin invertible}\}.$$

In \cite{Berk2002} it is shown that
$$\sigma_{BW}(T)=\displaystyle{\bigcap_{F\in
{\mathcal F}(X)}}\sigma_{D}(T+F).$$ From \cite[Lemma 4.1]{Berk2002},
$T$ is a B-Weyl operator if and only if $T=F\oplus N$, where $F$ is
a Weyl operator and $N$ is a nilpotent operator. Hence
$\sigma_{BW}(T)\subset\sigma_D(T)$. The defect set
$\sigma_D(T)\setminus\sigma_{BW}(T)$ has been  characterized in
\cite{AZ2011, zguitti2010} as follows:
\begin{equation}\label{DBWS}\sigma_{BW}(T)\cup[{\mathcal S}(T)\cap {\mathcal S}(T^*)]=\sigma_D(T).
\end{equation}

Quasi-Fredholm operators were generalized to pseudo Fredholm operators. In fact, $T\in L(X)$ is said to be a {\it pseudo Fredholm} operator if there exist two closed $T$-invariant  subspaces $X_1,$ $X_2$ such that $X=X_1\oplus X_2,$ $T|X_{1}$  is  quasi-nilpotent and  $T|X{_2}$ is  semi-regular. This decomposition is called the {\it generalized Kato decomposition}, see \cite{mbekhta1, mbekhta2}.

\smallskip
 Following Koliha
\cite{Kol1}, an operator $T\in L(X)$ is {\it generalized Drazin
invertible} if and only if $0\not\in \mbox{acc} \sigma(T),$ which is
also equivalent to the fact that $T=T_1\oplus T_2;$ where $T_1$ is
invertible and $T_2$ is quasi-nilpotent. The {\it generalized Drazin spectrum} of $T\in L(X)$  is defined by
$$\sigma_{gD}(T)=\{\lambda\in\mathbb{C} : T-\lambda I\mbox{ is not
generalized Drazin invertible}\}.$$
 For more details about
generalized Drazin invertibility, we refer the reader to
\cite{DjCzech, Kol1, Kol2}. It is not difficult to see that
$\sigma_{D}(T)=\sigma_{gD}(T)\cup\mbox{iso}\sigma_{D}(T).$ The
inclusion $\sigma_{gD}(T)\subset\sigma_{D}(T)$  may be strict.
Indeed, let $T$  defined on $ l^2(\mathbb{N})$  by $$T(x_1, x_2,
x_3, \ldots)=({1\over 2}x_2, {1\over 3}x_3, {1\over 4}x_4,
\ldots),$$ then it is clear that $T$ is quasi-nilpotent with
infinite ascent. Hence $\sigma_{gD}(T)=\emptyset$ and
$\sigma_{D}(T)=\{0\}.$

\smallskip
More recently, B-Fredholm operators were
generalized to pseudo B-Fredholm operators. Precisely,  $T\in L(X)$ is said to be a {\it pseudo B-Fredholm} operator if there exist two closed $T$-invariant  subspaces $X_1,$ $X_2$ such that $X=X_1\oplus X_2,$ $T|X_{1}$  is  quasi-nilpotent and  $T|X{_2}$ is  Fredholm, see \cite{boasso}.

\smallskip
 As a continuation in this direction, in the second section of the present work,  we  generalize  the B-Weyl operators and then the Weyl operators
to pseudo B-Weyl operator.  $T\in L(X)$ will be said to be
 pseudo B-Weyl operator if $T$ can be written as $T=T_1\oplus T_2;$ where $T_1$ is
  Weyl operator and $T_2$ is quasi-nilpotent. The corresponding spectrum will
   be denoted by $\sigma_{pBW}(T).$ Among other things, we prove that
    $$\sigma_{pBW}(T)\cup[{\mathcal S}(T)\cap{\mathcal
S}(T^*)]=\sigma_{gD}(T).$$  We prove also that
$$\displaystyle{\bigcap_{F\in{\mathcal F}(X)}}\sigma_{gD}(T+F)\subset\sigma_{pBW}(T).$$
  In the third section,
  we investigate the generalized Drazin spectrum of upper triangular operator matrices  $M_C=(\begin{smallmatrix}
   A & C \\
 0 & B
 \end{smallmatrix})$, where $A\in L(X)$, $B\in L(Y)$ and $C\in L(Y,X)$. After remarking that the inclusion
$\sigma_{gD}(M_C)\subset \sigma_{gD}(A)\cup\sigma_{gD}(B)$  is
proper, we investigate the defect set
$[\sigma_{gD}(A)\cup\sigma_{gD}(B)]\setminus\sigma_{gD}(M_C)$ in
connection with local spectral theory. Precisely,
we prove that $\sigma_{gD}(M_C)\cup[{\mathcal S}(A^*)\cap
 {\mathcal S}(B)]=\sigma_{gD}(A)\cup\sigma_{gD}(B)$ for all $C\in L(Y,
 X),$ and  we give   sufficient conditions on
$A$ and $B$ which ensure the equality
$\sigma_{gD}(M_C)=\sigma_{gD}(A)\cup\sigma_{gD}(B).$ We also
investigate the largest set of operators $C\in L(Y,X)$ for which the
last equality holds for all $A\in L(X)$ and $B \in L(Y).$

\section{On pseudo B-Weyl operators}

\begin{defn}\label{dfn1}\rm Let $T\in L(X).$ We say that $T$ is {\it pseudo B-Weyl }
if there exist two closed T-invariant subspaces $X_1,$ $X_2$ such
that $X=X_1\oplus X_2$,
 $T|X_{1}$  is  a Weyl operator  and  $T|X{_2}$
 is  a quasi-nilpotent operator. The pseudo B-Weyl spectrum  $\sigma_{pBW}(T)$ of $T$ is defined by
  $\sigma_{pBW}(T)=\{\lambda\in\mathbb{C}\,:\,T-\lambda I\mbox{ is not
  pseudo B-Weyl}\}.$\end{defn}

It is easy to see that $T$ is pseudo B-Weyl
if and only if $T^*$ is pseudo B-Weyl. Then
$\sigma_{pBW}(T)=\sigma_{pBW}(T^*).$ Let $pBW(X)$ denote the class
of all pseudo B-Weyl operators. From the definition of pseudo
B-Weyl operators, it is easily seen that all B-Weyl operators, all
 quasi-nilpotent operators and all generalized Drazin operators are
 pseudo B-Weyl operators. So the class $pBW(X)$  contains $BW(X)$ the class of B-Weyl operators as a proper subclass.

\begin{thm} Assume that $\mathcal{H}$ is a separable, infinite dimensional, complex Hilbert space. Then for every $T\in L(\mathcal{H})$ the following assertions are equivalent:\\
i) $T$ is in the norm closure of $pBW(\mathcal{H});$\\
ii) $T$ is in the norm closure of $BW(\mathcal{H}).$
\end{thm}
\begin{proof} (i) $\Longrightarrow$ (ii) because $BW(\mathcal{H})\subset pBW(\mathcal{H}).$\\
 (ii) $\Longrightarrow$ (i) Let $T\in pBW(\mathcal{H})$. Then $T=T_1\oplus T_2$ where $T_1$ is Weyl
  operator and $T_2$ is quasi-nilpotent.
 Then it follows from \cite{Apostol} that there exists a sequence of nilpotent operators $T_{2,n}$ which converges in
  norm to $T_2$. Hence $T_1\oplus T_{2,n}$ is a sequence of B-Weyl operators which converges in norm to $T.$
  Thus $T$ is in the norm closure of $BW(\mathcal{H}).$
\end{proof}

\begin{cor} Assume that $\mathcal{H}$ is a separable, infinite dimensional, complex Hilbert space. Then
 $$\overline{gBW(\mathcal{H})}^{\|\,\|}=\overline{BW(\mathcal{H})}^{\|\,\|}.$$
\end{cor}

\smallskip
Recall that $T\in L(X)$ is of {\it finite descent} if there exists a
 nonnegative integer $p$ such that  $\mathcal{R}(T^p)=\mathcal{R}(T^{p+1})$.

\begin{prop} Let $T\in L(X)$ with finite descent.
Then $T$ is pseudo B-Weyl if and only if $T$ is
B-Weyl.\end{prop}

\begin{proof} If  $T$ is pseudo B-Weyl , then $X=X_1\oplus X_2,$  where
$X_1,$ $X_2$ are closed subspaces of $X,$
 $T|X_1$ is Weyl operator and $T|X_2$ is quasi-nilpotent operator.
 Since $T$ is of finite descent, then $T|X_1$ and $T|X_2$ both are of finite descent.
Now $T|X_2$ is quasi-nilpotent with finite descent, then it follows from \cite[Corollary
10.6]{Tay} that $T|X_2$ is nilpotent operator. Thus
$T$ is B-Weyl operator by \cite[Lemma 4.1]{Berk2002}. The
opposite sense  is always true.\end{proof}

\begin{rem}\rm Let $T$ be a bilateral shift on $l^2(\mathbb{Z})$.
Then $T$ is pseudo B-Weyl if and only if $T$ is Weyl operator or $T$
 is quasi-nilpotent operator. Indeed, if $T$ is pseudo B-Weyl , then there exist two closed  T-invariant
 subspaces $X_1$ and $X_2$ such that $l^2(\mathbb{Z})=X_1\oplus X_2,$
 $T|X_1$ is Weyl operator and $T|X_2$ is quasi-nilpotent operator.
 Let $P$ be the projection on $X_1$ with ${\mathcal R}(P)=X_1$ and ${\mathcal N}(P)=X_2$.
 Since $P$ commutes with $T$ then by \cite[Theorem 3]{Shields} there
 exists some $\phi\in L^\infty(\beta)$ such that $M_\phi$ is similar to $P$.
 Since $P^2=P$ then $\phi^2=\phi$. Hence $\phi=1$ or $\phi=0$.
 Thus $P=I$ or $P=0$. Then $X_1=l^2(\mathbb{Z})$ or $X_1=\{0\}$.
 It follows that $T$ is Weyl or quasi-nilpotent. The converse is trivial.\end{rem}

It is easily seen that
$\sigma_{pBW}(T)\subset\sigma_{gD}(T).$ But,
in general,  this inclusion is proper as we can see in the
following example.

\begin{ex}\label{ex1}\rm
Here and elsewhere $S$ denotes the unilateral unweighted shift
operator on $l^2(\mathbb{N})$ defined by $$S(x_1, x_2, x_3,
\ldots)=(0, x_1, x_2, x_3, \ldots).$$ Let $T=S\oplus S^*.$ Then
$\sigma_{gD}(T)=\{\lambda\in\mathbb{C}\,:\,|\lambda|\leq 1\}$. As
$n(T)=d(T)=1$ then $T$ is pseudo B-Weyl. So
$0\not\in\sigma_{pBW}(T).$ This shows that the inclusion
$\sigma_{pBW}(T)\subset\sigma_{gD}(T)$ is proper. \end{ex}

 Then it
is naturel to ask about the   defect set
$\sigma_{gD}(T)\setminus\sigma_{pBW}(T)$. Thanks to the SVEP we give
 a characterization of  this  defect set.

\begin{thm}\label{thm1} Let $T\in L(X).$ Then
$$\sigma_{pBW}(T)\cup[{\mathcal S}(T)\cap{\mathcal
S}(T^*)]=\sigma_{gD}(T).$$\end{thm}

\begin{proof} Since $\sigma_{pBW}(T)\cup[{\mathcal S}(T)\cap{\mathcal
S}(T^*)]\subset\sigma_{gD}(T)$ always holds, let
$\lambda\in\sigma_{gD}(T)\setminus\sigma_{pBW}(T).$ Without loss of
generality we can assume that $\lambda=0$.
 Then $T=T_1\oplus T_2$ on $X=X_1\oplus X_2$ such that $T_1$ is Weyl operator
  and $T_2$ is quasi-nilpotent operator. Assume that $0\not\in {\mathcal S}(T)\cap{\mathcal S}(T^*)$.

{\it Case 1.} $0\notin {\mathcal S}(T)$: since $T$ has  SVEP at $0,$
 then $T_1$ also has  SVEP at $0.$ As $T_1$ is Weyl operator and
 then is B-Weyl, it follows from \cite[Theorem 2.3]{AZ2011} that
$T_1$ is Drazin invertible. Moreover, $0 \in\sigma(T_1),$ because in
the otherwise,  $T_1$ will be invertible
 and therefore $T$ is
generalized Drazin invertible, a contradiction. Hence $0\not\in
\mbox{acc} \sigma(T_1).$ Since $T_2$ is quasi-nilpotent then
$0\not\in \mbox{acc} \sigma(T).$ Thus $T$ is generalized Drazin
  invertible. Which leads a contradiction.

{\it Case 2.} $0\notin {\mathcal S}(T^*)$: the proof follows similarly.
\end{proof}

From Theorem \ref{thm1},  in the   following corollary, we explore
sufficient conditions which ensures the equalities
$\sigma_{pBW}(T)=\sigma_{gD}(T)$. We point out that for the operator
$T$ defined in  Example \ref{ex1} we have ${\mathcal S}(T)={\mathcal
S}(T^*) ={\mathcal S}(S^*)=\{\lambda \in\mathbb{C}:
0\leq|\lambda|<1\}$ (see for instance \cite{HZ2, LN}). Hence
${\mathcal S}(T)\cap{\mathcal S}(T^*) =\{\lambda \in\mathbb{C}:
0\leq|\lambda|<1\}.$

\begin{cor}\label{cgBWgd} Let $T\in L(X).$ If ${\mathcal S}(T)\cap {\mathcal S}(T^*)=\emptyset,$ then
 $$\sigma_{pBW}(T)=\sigma_{gD}(T).$$ In particular, the equality holds if $T$ or $T^*$ has SVEP.
\end{cor}

In the next proposition, we show that  generalized Drazin spectrum
is stable under quasi-nilpotent and finite rank  commuting
perturbations.

\begin{prop} Let $T\in L(X).$ The the following statements hold.\\
i) If  $F\in {\mathcal F}(X)$ and commutes with $T,$ then
$\sigma_{gD}(T+F)=\sigma_{gD}(T).$\\
ii) If $Q\in L(X)$ is  a quasi-nilpotent and commutes  with $T,$
then $\sigma_{gD}(T+Q)=\sigma_{gD}(T).$\end{prop}

\begin{proof}  (i) From \cite[Lemma
2.1]{YH} we know that $\mbox{acc}\sigma(T+F)=\mbox{acc}\sigma(T).$
 Then $\lambda\not\in \mbox{acc} \sigma(T+F)$ $\Longleftrightarrow$
$\lambda\not\in \mbox{acc} \sigma(T).$ Hence $T+F-\lambda I$ is
generalized Drazin invertible if and only if $T-\lambda I$ is
generalized Drazin invertible, as desired.\\
(ii) Since $\sigma(T+Q)=\sigma(T)$ then
$\mbox{acc}\,\sigma(T+Q)=\mbox{acc}\sigma(T).$ Thus $T+Q-\lambda I$
is generalized Drazin invertible $\Longleftrightarrow$  $T-\lambda
I$ is. So $\sigma_{gD}(T+Q)=\sigma_{gD}(T).$
\end{proof}

\begin{thm} Let $R,T,U\in L(X)$ be such that $TRT=TUT.$ Then $$\sigma_{gD}(TR)=\sigma_{gD}(UT).$$
\end{thm}

\begin{proof}
Since
$\sigma(TR)\setminus\{0\}=\sigma(UT)\setminus\{0\}$, by \cite[Theorem 1]{Corach}, then it is
enough to show that $TR$ is generalized Drazin invertible
$\Longleftrightarrow$  $UT$ is. Assume that
$0\not\in\sigma_{gD}(TR)$, then $0\not\in\mbox{acc}\sigma(TR).$
Therefore  $TR-\mu I$ is invertible for all small $\mu\neq 0.$ Hence
$UT-\mu I$ is
 invertible for all small $\mu\neq 0.$ So $0\not\in \mbox{acc} \sigma(UT).$ Hence $UT$
 is generalized Drazin invertible $\Longleftrightarrow$ $TR$ is. \end{proof}

 In particular if $R=U$ we get
\begin{cor}\label{ceq12} Let $R,T\in L(X)$ then $$\sigma_{gD}(TR)=\sigma_{gD}(RT).$$
\end{cor}

 Since the equality $\mathcal{S}(TR)=\mathcal{S}(UT)$ always holds (see \cite[Theorem
 9]{Corach}), then it follows from Theorem \ref{thm1} that
   $\sigma_{pBW}(TR)\cup[\mathcal{S}(TR)\cap\mathcal{S}(R^*T^*)]=\sigma_{pBW}(UT)
   \cup[\mathcal{S}(TR)\cap\mathcal{S}(R^*T^*)].$
 In particular we get from last theorem that for $R$ and $T\in L(X)$, $\sigma_{gD}(TR)=\sigma_{gD}(RT)\mbox{ and }\sigma_{pBW}(TR)\cup(\mathcal{S}
  (TR)\cap\mathcal{S}(R^*T^*))=\sigma_{pBW}(RT)\cup[\mathcal{S}(TR)\cap\mathcal{S}(R^*T^*)].$

\begin{thm} Let $T\in L(X)$. Then
$$ \displaystyle{\bigcap_{F\in {\mathcal
F}(X)}}\sigma_{gD}(T+F)\subset\sigma_{pBW}(T).$$ \end{thm}

\begin{proof} Let $\lambda\notin\sigma_{pBW}(T)$ arbitrary, then $T-\lambda I$ is pseudo B-Weyl
 operator. Therefore $X=X_1\oplus X_2$ and  $T-\lambda I=T_1\oplus T_2$
 relatively to this decomposition, with
 $T_1$ is  Weyl operator and $T_2$ is quasi-nilpotent operator.
  By \cite[Theorem 6.5.2]{Harte} there exists a finite rank operator
  $F_1$ such that $T_1+F_1$ is invertible. Let $F=F_1\oplus 0.$ Then
 $F$ is a finite rank operator, $(T-\lambda I)+F=(T_1+F_1)\oplus T_2$
 is generalized Drazin invertible and
 $\lambda\notin\displaystyle{\bigcap_{F\in {\mathcal F}(X)}}\sigma_{gD}(T+F).$ \end{proof}

We would like to finish this section with the following

\smallskip
 \noindent {\bf Question:} {\it Is it
 true that}
 $\sigma_{pBW}(T)=\displaystyle{\bigcap_{F\in
{\mathcal F}(X)}}\sigma_{gD}(T+F)?$\\

\section{Generalized Drazin invertibility for operator matrices}

For bounded linear operators $A\in L(X),B\in L(Y)$ and
$C\in L(Y,X),$ by $M_C$ we denote  the operator matrices
$
M_C = \left(\begin{array}{cc}
   A &C \\
   0 & B
\end{array}\right)$ defined  on $X\oplus Y.$

It is well known that, in the case of infinite dimensional, the
inclusion $\sigma(M_C)\subset\sigma(A)\cup\sigma(B)$ may be strict .
Hence several authors have been  interested by the defect set
$[\sigma_*(A)\cup\sigma_*(B)]\setminus\sigma_*(M_C)$ where
$\sigma_*$ runs different type spectra, see for instance
\cite{Barnes1, Boumaz, DZ, Dug, Lb, Lee, Zg2, ZZ, zguitti2010,
Zg2013, Zhang, Zhang2} and the references therein.

\smallskip
We begin this section by proving that the  generalized Drazin
spectrum of a direct sum is the union of generalized Drazin spectra
of its summands, and that  this result does not hold, in general,
for
 the generalized B-Weyl spectrum.

 \begin{prop}\label{prop2} Let $A\in L(X)$ and $B\in L(Y)$. Then
$$\sigma_{gD}(M_{0})=\sigma_{gD}(A)\cup\sigma_{gD}(B).$$\end{prop}

\begin{proof} Let $\lambda\not\in\sigma_{gD}(A)\cup\sigma_{gD}(B),$
then $\lambda\not\in\mbox{acc}\sigma(A)\cup\mbox{acc}\sigma(B).$ As
$\mbox{acc}\sigma(A)\cup\mbox{acc}\sigma(B)=\mbox{acc}[\sigma(A)\cup\sigma(B)],$
then $\lambda\not\in\mbox{acc}\sigma(M_{0}).$ So $\lambda\not\in
\sigma_{gD}(M_{0})$ and hence
$\sigma_{gD}(M_{0})\subset\sigma_{gD}(A)\cup\sigma_{gD}(B).$

Conversely, let $\lambda\not\in\sigma_{gD}(M_{0}),$ then
$\lambda\not\in\mbox{acc}\sigma(M_{0})=\mbox{acc}\sigma(A\oplus B).$
Therefore
$\lambda\not\in\mbox{acc}\sigma(A)\cup\mbox{acc}\sigma(B).$ Thus
$\lambda\not\in\sigma_{gD}(A)\cup\sigma_{gD}(B).$ Hence
$\sigma_{gD}(A)\cup\sigma_{gD}(B)\subset\sigma_{gD}(M_{0}).$ This
finishes the proof.\end{proof}

\begin{ex}\rm Let $R\in  L(X)$ and  $T\in L(X).$ Let $A$ be the operator defined on $X\oplus X$ by

$$A = \left(\begin{array}{cc}
   0 &T \\
 R& 0
\end{array}\right),$$  then $A^2=\left(\begin{array}{cc}
   TR &0 \\
 0& RT
\end{array}\right).$ Thus  it follows from the above proposition that $\sigma_{gD}(A^2)=\sigma_{gD}(TR)\cup\sigma_{gD}(RT)$, which equals to $\sigma_{gD}(TR)$ by Corollary \ref{ceq12}.
Therefore $\sigma_{gD}(A)=\sqrt{\sigma_{gD}(TR)}.$
\end{ex}

\begin{rem}\rm

In general,   the equality proved in Proposition \ref{prop2} for
 the generalized spectrum does not hold for the pseudo B-Weyl spectrum.
For this, let $S$ be the unilateral unweighted shift on
$l^2(\mathbb{N})$ and set $A=S$ and $B=S^*.$ Since $A$ and $B^*$
have  SVEP then
$\sigma(A)=\sigma_{gD}(A)=\sigma_{pBW}(A)=\sigma(B)=\sigma_{gD}(B)=\sigma_{pBW}(B)=
 \{\lambda\in\mathbb{C}: |\lambda|\leq 1\},$ while $0\not\in\sigma_{pBW}(M_0).$
\end{rem}

Generally, the study of generalized Drazin invertibility for upper
triangular operator matrices was firstly investigated by D. S.
Djordjevi\'c and P. S. Stanimirovi\'c \cite{DjCzech}. They proved in
particular that
\begin{eqnarray}\sigma_{gD}(M_C)\subset\sigma_{gD}(A)\cup\sigma_{gD}(B)\mbox{
for every }C\in L(Y,X).\label{eq1}\end{eqnarray} But this inclusion
may be strict as we can see in  the following example.

\begin{ex}\label{ex2}\rm  Let $A=S$
 be the unilateral shift on $l^2(\mathbb{N})$ and let $B=S^*$ and
  $ C = I- SS^*.$ Then $M_C$ is unitary and hence we get
  $$\sigma_{gD}(M_C)=\{\lambda\in\mathbb{C} :|\lambda|=1\}
   \mbox{ and } \sigma_{gD}(A)\cup\sigma_{gD}(B)=\{\lambda\in\mathbb{C} : 0\leq |\lambda|\leq 1\}.$$

 \end{ex}

The defect set
$(\sigma_{gD}(A)\cup\sigma_{gD}(B))\setminus\sigma_{gD}(M_C)$ has
been studied very recently in \cite{Zhang2}, more precisely, it was
proved that this defect is the union of certain holes in
$\sigma_{gD}(M_C)$ which happen to be subsets of
$\sigma_{gD}(A)\cap\sigma_{gD}(B)$. We will explicit in what follows
the defect set $[\sigma_{gD}(A)\cup\sigma_{gD}(B)]\setminus
\sigma_{gD}(M_C)$ by means of localized SVEP. This result will lead
us to a necessary condition that ensures the equality desired (see
Corollary \ref{cdSVEP} bellow).

\begin{thm}\label{thm3} For $A\in L(X),\,B\in  L(Y)$ and $C\in L(Y,X)$
 we have
 $$\sigma_{gD}(M_C)\cup[{\mathcal S}(A^*)\cap
 {\mathcal S}(B)]=\sigma_{gD}(A)\cup\sigma_{gD}(B).$$ \end{thm}

\begin{proof}  Let
$\lambda\in(\sigma_{gD}(A)\cup\sigma_{gD}(B))\setminus\sigma_{gD}(M_C).$
We can assume   without loss of generality  that $\lambda=0.$ Then
$M_C$ is generalized Drazin invertible and hence $0\not\in
\mbox{acc} \sigma(M_C).$ Then there exists $\varepsilon>0$ such that
 $M_C-\mu I$ is invertible for every  $0<|\mu|<\varepsilon.$  Thus
 for every  $0<|\mu|<\varepsilon,$
$A-\mu I$ is left invertible and $B-\mu I$ is right invertible. So
$0\not\in \mbox{acc}\sigma_{ap}(A)\cup \mbox{acc}\sigma_s(B).$ For
the sake of contradiction assume that $0\notin {\mathcal S}(A^*)\cap
 {\mathcal S}(B).$

 {\it Case 1.}  $0\not\in {\mathcal S}(A^*)$: If
$0\in\sigma(A^*)$ then since $\sigma(A^*)={\mathcal S}(A^*)\cup\sigma_{s}(A^*)$ we have  $0\in\sigma_s(A^*).$ As
$0\not\in \mbox{acc}\sigma_{ap}(A)=\mbox{acc}\sigma_s(A^*),$ then
$0$ is an isolated point of $\sigma_s(A^*).$ Thus  $0$ is an
isolated point of $\sigma(A^*)=\sigma(A).$ Hence  $A$ is generalized
Drazin invertible. Summing up: $M_C$ and $A$ are generalized Drazin
invertible, which implies from \cite[Lemma 2.5]{Zhang2} that $B$ is
also generalized Drazin invertible. But this is impossible. Now if
$0\not\in\sigma(A^*)$ then $0\not\in\sigma_{gD}(A).$ Hence $B$ will
be  generalized Drazin invertible, and this is a contradiction.

 {\it
Case 2.} $0\not\in {\mathcal S}(B)$: If $0\not\in\sigma(B)$ then
$0\not\in\sigma_{gD}(B).$ So $B$ is generalized Drazin invertible,
and since $M_{C}$ is generalized Drazin invertible it follows from
\cite[Lemma 2.5]{Zhang2} that  $A$ is generalized Drazin invertible.
But this is a contradiction. If $0\in\sigma(B)$ then
$0\in\sigma_s(B).$ As $0\not\in \mbox{acc}\sigma_s(B)$ then
$0\in\mbox{iso}\sigma_s(B),$ therefore $0\in\mbox{iso}\sigma(B).$ So
$B$ is  generalized Drazin invertible and since $M_{C}$ is
generalized Drazin invertible, it then follows that $A$ is
generalized Drazin invertible. But this is
a contradiction.\\
In the two cases we have
$\sigma_{gD}(A)\cup\sigma_{gD}(B)\subset\sigma_{gD}(M_C)\cup[{\mathcal S}(A^*)\cap{\mathcal S}(B)].$ Since the opposite inclusion is always
true then
$\sigma_{gD}(A)\cup\sigma_{gD}(B)=\sigma_{gD}(M_C)\cup[{\mathcal S}(A^*)\cap{\mathcal S}(B)].$ Hence the theorem is
proved.\end{proof}

 Now, in the next corollary, we give a sufficient condition which ensures that
$\sigma_{gD}(M_C)=\sigma_{gD}(A)\cup\sigma_{gD}(B)$ for every $ C\in L(Y,X).$ We notice that the condition
 ${\mathcal S}(A^*)\cap
 {\mathcal S}(B)=\emptyset$ is not satisfied  for
   operators $A$ and $B$ defined in  Example \ref{ex2}.

\begin{cor}\label{cdSVEP}Let $A\in L(X)$ and Let $B\in L(Y).$ If ${\mathcal S}(A^*)\cap
 {\mathcal S}(B)=\emptyset$
then for every $ C\in L(Y,X),$
$\sigma_{gD}(M_C)=\sigma_{gD}(A)\cup\sigma_{gD}(B).$ In particular, if
$A^*$ or $B$ has  SVEP, then $\sigma_{gD}(M_C)=\sigma_{gD}(A)\cup\sigma_{gD}(B).$ \end{cor}

\begin{ex}\rm Let $S$ be the unilateral shift operator on $l^2(\mathbb{N})$ and  we define operators
 $A=(S\oplus S^*)+I$ and $B=(S\oplus S^*)-I$ on $l^2(\mathbb{N})\oplus l^2(\mathbb{N}).$
 Then $$\sigma(A)=\{\lambda\in\mathbb{C} : 0\leq |\lambda-1|\leq 1\} \mbox{ , }
  \sigma(B)=\{\lambda\in\mathbb{C} : 0\leq |\lambda+1|\leq 1\}.$$ It
  follows that
$${\mathcal S}(A)=\{\lambda\in\mathbb{C} : 0\leq|\lambda-1|< 1\} \mbox{ , } {\mathcal S}(B)=\{\lambda\in\mathbb{C} : 0\leq|\lambda+1|< 1\}.$$
Hence  ${\mathcal S}(A^*)\cap {\mathcal S}(B)=\emptyset.$  Therefore
$\sigma_{gD}(M_C)=\sigma_{gD}(A)\cup\sigma_{gD}(B).$
 Note here that $A^*$ and $B$ do not have  SVEP.
\end{ex}

The equality $\sigma_{gD}(M_C)=\sigma_{gD}(A)\cup\sigma_{gD}(B)$
holds in particular, if we take $A=S^*$ or $B=S,$ since in this case
$A^*$ or $B$ has  SVEP. It also holds when $A$ and $B$ belong to the class of all normal or hyponormal operators in  Hilbert spaces, or the class of all compact operators in Banach spaces.

\begin{rem}\rm  Generally,  we do not have
$\sigma_{pBW}(M_C)=\sigma_{pBW}(A)\cup\sigma_{pBW}(B)$ even if $A^*$
or $B$ has SVEP. For instance, let $S$ be the unilateral unweighted
shift on $l^2(\mathbb{N})$. Let $A=S^*, $ $B=S$ and $C=I-SS^*.$
Since $A^*$ and $B$ have  SVEP, it follows from Corollary
\ref{cgBWgd} that $\sigma_{pBW}(A)=\sigma_{gD}(A)$ and
$\sigma_{pBW}(B)=\sigma_{gD}(B).$ Hence
$\sigma_{pBW}(A)\cup\sigma_{pBW}(B)=\{\lambda\in\mathbb{C} : 0\leq
|\lambda|\leq 1\}.$ Since $M_C$ is unitary then
$\sigma_{pBW}(M_C)=\{\lambda\in\mathbb{C} : |\lambda|= 1\}.$ Thus
$\sigma_{pBW}(M_C)\neq\sigma_{pBW}(A)\cup\sigma_{pBW}(B).$ Here $A$
and $B^*$ do not have  SVEP. However, we have the following
result.\end{rem}

\begin{prop} Let $A\in L(X)$ and  $B\in L(Y).$ If  $A$ and $B$ (or $A^*$ and $B^*$) have SVEP,
then for every $ C\in L(Y,X),$
$$\sigma_{pBW}(M_C)=\sigma_{pBW}(A)\cup\sigma_{pBW}(B).$$\end{prop}

\begin{proof} If $A$ and $B$ have  SVEP then $M_C$ has also  SVEP, see \cite[Proposition 3.1]{Zg2}. Hence
\begin{eqnarray*}
 \sigma_{pBW}(M_C)&=&\sigma_{gD}(M_C)\,\,\mbox{ (by Corollary \ref{cgBWgd})}\\
                 &=&\sigma_{gD}(A)\cup\sigma_{gD}(B)\,\,\mbox{ (by Corollary \ref{cdSVEP})}\\
                &=&\sigma_{pBW}(A)\cup\sigma_{pBW}(B)\,\,\mbox{ (by Corollary \ref{cgBWgd})}.
\end{eqnarray*}
The case of $A^*$ and $B^*$ have SVEP goes similarly.\end{proof}

In our next result, we are going to provide a new condition under which the
equality $\sigma_{gD}(M_C)=\sigma_{gD}(A)\cup\sigma_{gD}(B)$ holds.

\begin{prop}\label{p3.1} Let $A\in L(X),$ $B\in L(Y)$
 and $ C\in L(Y,X).$ If
  $\sigma_{pBW}(M_C)=\sigma_{pBW}(A)\cup\sigma_{pBW}(B),$ then
   $\sigma_{gD}(M_C)=\sigma_{gD}(A)\cup\sigma_{gD}(B).$\end{prop}

\begin{proof} Let $\lambda\notin\sigma_{gD}(M_C)$ arbitrary, then $M_C-\lambda I$ is generalized
 Drazin invertible. Hence  $\lambda\not\in\sigma_{pBW}(M_{C})=\sigma_{pBW}(A)\cup\sigma_{pBW}(B).$
  So $A-\lambda I$ and
  $B-\lambda I$ are pseudo B-Weyl operators. If
  $\lambda\in\sigma_{gD}(A)$ then
   form Theorem \ref{thm1} we have  $\lambda\in{\mathcal S}(A)\cap{\mathcal S}(A^*).$
  Hence $\lambda\in{\mathcal S}(A) \subset{\mathcal S}(M_C)\subset\sigma_{gD}(M_C).$
  But  this is  a contradiction.
  Therefore $\lambda\not\in\sigma_{gD}(A).$ From \cite[Lemma 2.5]{Zhang2} we conclude  that
  $\lambda\not\in\sigma_{gD}(B).$ Thus
  $\lambda\not\in\sigma_{gD}(A)\cup\sigma_{gD}(B)$. Hence
  $\sigma_{gD}(M_C)\supseteq\sigma_{gD}(A)\cup\sigma_{gD}(B).$
  Since $\sigma_{gD}(M_C)\subset\sigma_{gD}(A)\cup\sigma_{gD}(B)$ holds  with no restriction,
  then  $\sigma_{gD}(M_C)=\sigma_{gD}(A)\cup\sigma_{gD}(B).$\end{proof}

 One might expect that the converse  of Proposition \ref{p3.1} is true, but this is not true in general as shown in the following example.

\begin{ex}\rm Let $S$ be the unweighted unilateral shift on $l^2(\mathbb{N})$.
On $l^2(\mathbb{N})\otimes l^2(\mathbb{N})$ set $A=S\otimes I$,  $B=S^*\otimes I$ and

\[ C = \left(\begin{array}{cccc} 0 &  &  & \\
                                  & I-SS^* & & \\
                                   & & I-SS^* & \\
                                    & & & \ddots
                                                  \end{array}  \right).\]
Then $\sigma(M_C)=\{\lambda\in\mathbb{C} : 0\leq |\lambda|\leq
1\}=\sigma(A)=\sigma(B).$ Hence
$\sigma_{gD}(M_C)=\{\lambda\in\mathbb{C} : 0\leq |\lambda|\leq
1\}=\sigma_{gD}(A)\cup\sigma_{gD}(B).$ Since $A$ and $B^*$ have
SVEP,
 then it follows from Corollary \ref{cgBWgd} that $\sigma_{pBW}(A)\cup\sigma_{pBW}(B)=\sigma_{gD}(A)\cup\sigma_{gD}(B),$
 while $\sigma_{pBW}(M_C)\subset \sigma_{W}(M_C)=\{\lambda\in\mathbb{C} : |\lambda|=1\}.$
\end{ex}

 The following result gives necessary and sufficient condition under which
the generalized Drazin spectrum of the  operator  $M_{C}$ is the
union of generalized Drazin spectra of its diagonal entries.

\begin{prop}\label{proples3} Let $A\in L(X),\,B\in L(Y)$ and $C\in L(Y,X).$
 Then the following assertions are equivalent.\\
 i) $\sigma(M_C)=\sigma(A)\cup\sigma(B)$;\\
 ii) $\sigma_{D}(M_C)=\sigma_{D}(A)\cup\sigma_{D}(B)$;\\
 iii) $\sigma_{gD}(M_C)=\sigma_{gD}(A)\cup\sigma_{gD}(B).$
\end{prop}

\begin{proof} For (i) $\Longleftrightarrow$ (ii) see \cite[Proposition 3.7]{Zg2013}.

 i) $\Leftrightarrow$ iii) was proved in \cite{Zhang2} but we give here another proof by using the local spectral property SVEP. Assume that $\sigma(M_C)=\sigma(A)\cup\sigma(B)$. Then it follows from \cite[Theorem 2.5]{Lb}
that  ${\mathcal S}(A^*)\cap
 {\mathcal S}(B)\subset \sigma(M_{C}).$  Since ${\mathcal S}(A^*)\cap {\mathcal S}(B)$
  is an open subset then ${\mathcal S}(A^*)\cap
 {\mathcal S}(B)\subset \mbox{acc}\sigma(M_C).$
 Thus ${\mathcal S}(A^*)\cap{\mathcal S}(B)\subset\sigma_{gD}(M_C).$
By Theorem \ref{thm3} we conclude that
$\sigma_{gD}(M_C)=\sigma_{gD}(A)\cup\sigma_{gD}(B).$

 Conversely, suppose that  $\sigma_{gD}(M_C)=\sigma_{gD}(A)\cup\sigma_{gD}(B).$ Then it follows from
  Theorem \ref{thm3} that ${\mathcal S}(A^*)\cap{\mathcal S}(B)\subset\sigma_{gD}(M_C).$
 So ${\mathcal S}(A^*)\cap{\mathcal S}(B)\subset\sigma(M_C).$
 This  implies again by \cite[Theorem 2.5]{Lb} that   $\sigma(M_C)=\sigma(A)\cup\sigma(B).$\end{proof}

\begin{cor} Let $A\in L(X),\,B\in L(Y).$  If $\sigma(A)=\sigma_l(A)$ or $\sigma(B)=\sigma_r(B)$
 then  for every $ C\in L(Y,X)$ we have  $$\sigma_{gD}(M_C)=\sigma_{gD}(A)\cup\sigma_{gD}(B).$$ \end{cor}

\begin{proof} If $\sigma(A)=\sigma_l(A)$ or $\sigma(B)=\sigma_r(B)$
 then  $\sigma(M_C)=\sigma(A)\cup\sigma(B),$ see
 \cite{Barnes1}. But this is equivalent from
  Proposition \ref{proples3} to say that $\sigma_{gD}(M_C)=\sigma_{gD}(A)\cup\sigma_{gD}(B).$ \end{proof}

For $U$ and $V\in L(X),$ let  $L_U$ (resp., $R_V$) be the left
(resp., right) multiplication operator given by $L_U(W) = UW$
(resp., $R_V(W) = WV$) and let $ \delta_{U,V} = L_U - R_V$ be the
usual generalized derivation associated with $U$ and $V.$ Let
$\mathcal{N}^{\infty}(U)=\displaystyle\bigcup_{n\geq
1}\mathcal{N}(U^n)$ denote the generalized kernel of $U.$

\begin{thm}\label{thmD1}Let $A\in{\mathcal L}(X)$ and $B\in{\mathcal L}(Y).$ If $C$ is in the closure of the
set $${\mathcal R}(\delta_{A,B})+{\mathcal N}(\delta_{A,B}) +
\bigcup_{\lambda \in\mathbb{C}}\mathcal{N}^{\infty}(L_{A-\lambda I})
+\bigcup_{\lambda \in \mathbb{C}}\mathcal{N}^{\infty}(R_{B-\lambda
I}),$$ then
$$\sigma_{gD}(M_C)=\sigma_{gD}(A)\cup\sigma_{gD}(B).$$\end{thm}

\begin{proof} If $C$ is in the closure of the set ${\mathcal R}(\delta_{A,B})+{\mathcal N}(\delta_{A,B}) +
\displaystyle\bigcup_{\lambda
\in\mathbb{C}}\mathcal{N}^{\infty}(L_{A-\lambda I})
+\bigcup_{\lambda \in \mathbb{C}}\mathcal{N}^{\infty}(R_{B-\lambda
I}),$ then $\sigma_{D}(M_C)=\sigma_{D}(A)\cup\sigma_{D}(B),$ see
\cite[Theorem 3.4]{Zg2013}. The result follows at once from Proposition \ref{proples3}.\end{proof}

In general there is no definite relation between the condition considered in Corollary \ref{cdSVEP} and the condition considered in the above theorem. Indeed, let $A=S$ and $B=S^*$. Set $C=A-B$, then $C\in {\mathcal R}(\delta_{A,B})$. Hence $\sigma_{gD}(M_C)=\sigma_{gD}(A)\cup\sigma_{gD}(B)$ while $\mathcal{S}(A^*)\cap\mathcal{S}(B)=\{\lambda\in\mathbb{C}\,:\,|\lambda|<1\}$.
\smallskip

 Now, in the following definition, we introduce the concept of right
 and left generalized Drazin invertibility for bounded linear
 operators.
\begin{defn}\rm Let $T\in L(X).$ We will say that \\
i) $T$ is {\it left generalized Drazin invertible} if $0\not\in
\mbox{acc}
\sigma_l(T).$ \\
ii)  $T$ is {\it right generalized Drazin invertible}
 if $0\not\in \mbox{acc} \sigma_r(T).$\\
The  {\it right generalized Drazin
  spectrum}   is defined by
 $$\sigma_{rgD}(T)=\{\lambda\in\mathbb{C} :  T-\lambda I \mbox{ is not right generalized Drazin invertible}\}$$
  and the  {\it left  generalized Drazin
  spectrum}   is defined by $$\sigma_{lgD}(T)=\{\lambda\in\mathbb{C} :  T-\lambda I \mbox{ is not left generalized Drazin invertible}\}.$$\end{defn}

\begin{thm}\label{thm4} Let $A\in L(X),$ $B\in L(Y)$ and $C\in L(Y,X).$
 If  $M_C$ is generalized Drazin invertible,  then the following statements hold.\\
i) $A$ is left generalized Drazin invertible.\\
ii) $B$ is right generalized Drazin invertible.\\
iii) There exists a constant $\gamma>0$ such that $d(A-\lambda
I)=n(B-\lambda I)$ for every $0<|\lambda|<\gamma.$\end{thm}

\begin{proof} Assume that $M_C$ is generalized Drazin invertible.
Then there exists  $\gamma>0$ such that $M_C-\lambda I$ is
invertible for every $0<|\lambda|<\gamma.$ So and by virtue  of
\cite[Theorem 2]{Lee} we have  $A-\lambda I$ is left invertible and
$B-\lambda I$
 is right invertible  for every $0<|\lambda|<\gamma.$ Thus
  $0\not\in \mbox{acc} \sigma_l(A)\cup\mbox{ acc} \sigma_r(B).$
 This proves that  $A$ is left generalized Drazin invertible and $B$ is right generalized Drazin invertible.
 On the other hand,  since $M_C-\lambda I$ is invertible for $0<|\lambda|<\gamma,$ then
 again by  \cite[Theorem 2]{Lee} we obtain that  $d(A-\lambda I)=n(B-\lambda I)$ for every
  $0<|\lambda|<\gamma.$ \end{proof}

From Theorem \ref{thm4}, we derive  the following corollary.

\begin{cor}\label{cor1.1} Let $A\in L(X),$ $B\in L(Y).$ Then for every $C\in
L(Y,X)$ we have
$$[\sigma_{gD}(A)\cup\sigma_{gD}(B)]\setminus[\sigma_{rgD}(A)\cap\sigma_{lgD}(B)]\subset\sigma_{gD}(M_C).$$
\end{cor}

\begin{proof} Let $\lambda\in[\sigma_{gD}(A)\cup\sigma_{gD}(B)]\setminus\sigma_{gD}(M_C),$
then $A-\lambda I$ is left generalized Drazin invertible and
$B-\lambda I$ is right generalized Drazin invertible, by Theorem
\ref{thm4}. If $\lambda\not\in\sigma_{rgD}(A),$ then $A-\lambda I$
is generalized Drazin invertible. Since $M_C-\lambda I$ is
generalized Drazin invertible, then $B-\lambda I$
 is also generalized Drazin invertible. This is a contradiction.
 Analogously, we have $\lambda\in\sigma_{lgD}(A).$
  Thus  $\lambda\in\sigma_{rgD}(A)\cap\sigma_{lgD}(B).$
\end{proof}

The following theorem gives a slight generalization of the main result of \cite{Zhang2}.

\begin{thm}\label{thmhole} Let $A\in L(X),$ $B\in L(Y).$ Then for every $C\in
L(Y,X)$ we have
$$\sigma_{gD}(A)\cup\sigma_{gD}(B)=\sigma_{gD}(M_C)\cup\mathcal{W},$$
where $\mathcal{W}$ is the union of certain holes in
$\sigma_{gD}(M_C)$ which happen to be subsets of
$\sigma_{rgD}(A)\cap\sigma_{lgD}(B).$
\end{thm}

\begin{proof} From \cite{Zhang2} we have $$\eta(\sigma_{gD}(A)\cup\sigma_{gD}(B))=\eta(\sigma_{gD}(M_C),$$
where $\eta(.)$ is the polynomially convex hull. By Corollary
\ref{cor1.1}
 the filling some holes in $\sigma_{gD}(M_C)$ should occurs in $\sigma_{rgD}(A)\cap\sigma_{lgD}(B).$
\end{proof}

\end{document}